\documentclass[leqno]{article}
\usepackage{geometry}
\usepackage{graphicx}	
\usepackage[cp1251]{inputenc}
\usepackage[english]{babel}
\usepackage{mathtools}
\usepackage{amsfonts,amssymb,mathrsfs,amscd,amsmath,amsthm}

\usepackage{verbatim}

\usepackage{url}

\usepackage{stmaryrd}

\usepackage{cmap}

\def\ig#1#2#3#4{\begin{figure}[!ht]\begin{center}%
\includegraphics[height=#2\textheight]{#1.eps}\caption{#4}\label{#3}%
\end{center}\end{figure}}

\def\thtext#1{
  \catcode`@=11
  \gdef\@thmcountersep{. #1}
  \catcode`@=12
}

\def\threst{
  \catcode`@=11
  \gdef\@thmcountersep{.}
  \catcode`@=12
}

\theoremstyle{plain}

\newtheorem{thm}{Theorem}[section]
\newtheorem{prop}[thm]{Proposition}
\newtheorem{cor}[thm]{Corollary}

\theoremstyle{definition}

\newtheorem{rk}[thm]{Remark}
\newtheorem{examp}[thm]{Example}

 \catcode`@=11
 \def\.{.\spacefactor\@m}
 \catcode`@=12

\def\R{\mathbb R}

\def\a{\alpha}

\def\e{\varepsilon}
\def\dl{\delta}
\def\D{\Delta}
\def\g{\gamma}

\def\l{\lambda}

\def\s{\sigma}

\def\0{\emptyset}
\def\:{\colon}
\def\<{\langle}
\def\>{\rangle}
\def\[{\llbracket}
\def\]{\rrbracket}

\def\rom#1{\emph{#1}}
\def\({\rom(}
\def\){\rom)}

\def\ss{\subset}

\def\sp{\supset}

\def\x{\times}

\def\bcAD{\overline{\mathstrut\cAD}}

\def\CARD{\operatorname{CARD}}

\def\diam{\operatorname{diam}}

\def\dis{\operatorname{dis}}

\def\Ext{\operatorname{Ext}}

\def\mst{\operatorname{mst}}
\def\MST{\operatorname{MST}}

\def\cA{{\cal A}}
\def\cAD{\mathcal{AD}}

\def\cD{{\cal D}}

\def\cH{{\cal H}}

\def\cM{{\cal M}}

\def\cP{{\cal P}}
\def\cR{{\cal R}}

\usepackage{epigraph}

\begin{document}
\title{Metric Space Recognition by Gromov--Hausdorff Distances to Simplexes}
\author{A.\,O.~Ivanov, E.\,S.~Lychagina, and A.\,A.~Tuzhilin}
\date{}
\maketitle

\begin{abstract}
In the present paper a distinguishability of bounded metric spaces by the set of the Gromov--Hausdorff distances to so-called simplexes (metric spaces with unique non-zero distance) is investigated. It is easy to construct an example of non-isometric metric spaces such that the Gromov--Hausdorff distance between them vanishes. Such spaces are non-distinguishable, of course. But we give examples of non-distinguishable metric spaces, the Gromov--Hausdorff distance between which is non-zero. More over, we prove several sufficient and necessary conditions for metric spaces to be non-distinguishable.

\end{abstract}

\setlength{\epigraphrule}{0pt}

\section{Introduction}
\markright{\thesection.~Introduction}
Modern computer graphics, image comparison and recognition systems, which are actively developing at present, often use so-called ``spaces of spaces''. In addition to these applications, these spaces have a theoretical significance, and therefore continue to attract the attention of specialists in pure and applied mathematics. One of the possible approaches to the study of such spaces is based on construction of a distance function between spaces as a ``measure of their dissimilarity''. In~1914, F.~Hausdorff~\cite{Hausdorff} introduced a non-negative symmetric function on pairs of non-empty subsets of a metric space $X$, equal to the greatest lower bound of such numbers $r$ that one set is contained in the $r$-neighborhood of the other and vice versa. This function turned out to be a metric on the space of closed bounded subsets of $X$. Later, in~1975, D.~Edwards~\cite{Edwards} and, independently, in~1981, M.~Gromov~\cite{Gromov} generalized the Hausdorff construction to the family of all metric compacta, using their isometric embeddings into all possible ambient spaces (see the definition below). The resulting function is called the Gromov--Hausdorff distance, and the corresponding metric space of metric compacta, considered up to isometry, is called the Gromov--Hausdorff space and is denoted by $\cM$. The geometry of this space turns out to be rather bizarre and is actively studied. It is well known that $\cM$ is a complete, separable, geodesic space, and also that $\cM$ is not proper~\cite{BurBurIva,IvaNikTuz}.

The structure of the Gromov--Hausdorff space became the subject of the present paper. Pursuing the idea of ``introducing coordinates in the space $\cM$'', a family of metric spaces with one non-zero distance, which we will call simplexes, was proposed for the role of coordinate axes. More precisely, each axis is a set of simplexes of fixed cardinality (but different diameters), and a metric compact is associated with a set of distance functions from it to the points of the axis. The problem of calculating the Gromov--Hausdorff distance between arbitrary metric spaces is quite nontrivial, therefore this choice of coordinate axes is due to the relative simplicity of such distance functions, a detailed study of which is given in~\cite{GrigIvaTuzSimpDist}.

In the paper~\cite{IvaTuzDistSympl} an example of two finite metric spaces that do not differ in distances to finite simplexes was given (see also below). From the results of~\cite{GrigIvaTuzSimpDist} it immediately follows that the distances from these spaces to all possible simplexes are the same. The question arises: which spaces can be distinguished from each other in this way? We will give a number of necessary and sufficient conditions for the indistinguishability of spaces. We will show that the cardinalities of finite indistinguishable spaces must coincide. An example of a finite and a continuum indistinguishable spaces will be given, as well as a series of examples of indistinguishable spaces among which there are pairs of all possible cardinalities not less than the continuum.

\section{Preliminaries}\label{sec:GH}
\markright{\thesection.~Preliminaries}
Let us recall the necessary concepts and results concerning the Hausdorff and Gromov--Hausdorff distances. More detailed information can be found in~\cite{BurBurIva}.

Let $X$ be an arbitrary non-empty set. Denote by $\#X$ the cardinality of $X$. If $X$ is a metric space and $x$ and $y$ are its points, then by $|xy|$ we denote the distance between these points, and if $A$ and $B$ are non-empty subsets of $X$, then we set $|AB|=|BA|=\inf\bigl\{|ab|:a\in A,\,b\in B\bigr\}$. If $A=\{a\}$, then instead of $\bigl|\{a\}B\bigr|=\bigl|B\{a\}\bigr|$ we write $|aB|=|Ba|$. For a metric space $X$, a point $x\in X$ and a real number $r>0$, by $U_r(x)$ we denote the open ball with center at $x$ and radius $r$, and for a non-empty subset $A\ss X$ its $r$-neighborhood is defined as follows:
$$
U_r(A)=\big\{x\in X:|Ax|<r\big\}=\cup_{a\in A}U_r(a).
$$

Let $X$ be a set. By $\cP_0(X)$ we denote the family of all its non-empty subsets. Further, let $X$ be a metric space, and  $A,B\in\cP_0(X)$. The value
$$
d_H(A,B)=\inf\bigl\{r\in[0,\infty]:A\ss U_r(B)\quad\text{and}\quad U_r(A)\sp B\bigr\}
$$
is called the \emph{Hausdorff distance\/} between $A$ and $B$. It is well-known that $d_H$ is a metric on the set $\cH(X)$ of all non-empty closed bounded subsets of the metric space $X$.

Further, let $X$ and $Y$ be metric spaces. A triple $(X',Y',Z)$ consisting of a metric space $Z$ and its two subsets $X'$ and $Y'$ that are isometric to $X$ and $Y$, respectively, is called a \emph{realization of the pair $(X,Y)$}. The greatest lower bound $d_{GH}(X,Y)$ of reals $r$ such that there exists a realization $(X',Y',Z)$ of the pair $(X,Y)$ satisfying the inequality $d_H(X',Y')\le r$ is called the \emph{Gromov--Hausdorff distance\/} between the metric spaces $X$ and $Y$. It is well-known that the function $d_{GH}$ is a generalized pseudometric on the proper class of all metric spaces considered up to an isometry. This means that $d_{GH}$ can take infinite values ??and also be zero on a pair of non-isometric metric spaces. However, on the family of compact metric spaces, considered up to isometry, $d_{GH}$ is a metric.

For specific calculations of the Gromov--Hausdorff distance, the concept of distortion of generalized mappings --- binary relations --- turns out to be useful. A multivalued surjective mapping $R$ from a set $X$ to a set $Y$ is called a \emph{correspondence}. The set of all correspondences between $X$ and $Y$ is denoted by $\cR(X,Y)$. Thus, a correspondence is a binary relation between $X$ and $Y$ for which the natural projections onto $X$ and $Y$ are surjective.

For metric spaces $X$ and $Y$ and a binary relation $\s\in\cP_0(X\x Y)$ the value
$$
\dis\s:=\sup\Bigl\{\bigl||xx'|-|yy'|\bigr|:(x,y),\,(x',y')\in\s\Bigr\}
$$
is called the \emph{distortion\/} of the binary relation $\s$. The next result is well-known, see, for example,~\cite{BurBurIva}.

\begin{thm}\label{thm:GH-metri-and-relations}
For any non-empty metric spaces $X$ and $Y$ the following equality holds\/\rom:
$$
d_{GH}(X,Y)=\frac12\inf\bigl\{\dis R:R\in\cR(X,Y)\bigr\}.
$$
\end{thm}

For a cardinal number $n$, let $\D_n$ denote the metric space of cardinality $n$ in which all nonzero distances are equal to $1$. Note that $\D_1$ is a one-point metric space. For a metric space $X$ and a positive real number $\l$, let $\l X$ denote the metric space obtained from $X$ by multiplying all distances by $\l$. If $X$ is bounded and $\l=0$, then let $\l X=\D_1$. Clearly, $\l\D_1=\D_1$ for every $\l\ge0$. The spaces $\l\D_n$ are referred as \emph{simplexes of cardinality $n$}.

For a metric space $X$, let $\diam X$ denote its diameter, i.e., the value $\inf\bigl\{|xy|:x,y\in X\bigr\}$. Then it is easy to see that $2d_{GH}(\D_1,X)=\diam X$, see~\cite{BurBurIva}.

\section{Indistinguishable Metric Spaces}
\markright{\thesection.~Indistinguishable Metric Spaces}

We start with the following result.

\begin{thm}[\cite{GrigIvaTuzSimpDist}]\label{thm:BigSympl}
Let $X$ be a metric space, $\l\ge0$ a real number, and let $n$ be a cardinal number such that $\#X<n$, then
$$
2d_{GH}(\l\D_n,X)=\max\{\l,\,\diam X-\l\}.
$$
\end{thm}

In the present paper, we call metric spaces \emph{indistinguishable\/} if the Gromov--Hausdorff distances from them to each simplex are the same. Otherwise, we call the spaces \emph{distinguishable}. Theorem~\ref{thm:BigSympl} immediately implies the following result.

\begin{cor}\label{cor:diam}
If metric spaces are indistinguishable, then their diameters are the same.
\end{cor}

\begin{rk}
It also follows from Theorem~\ref{thm:BigSympl} that metric spaces $X$ and $Y$ of the same diameter can be distinguished only by simplexes, whose cardinality is less than or equals to $\max\{\#X,\#Y\}$.
\end{rk}

The following results concern the computation of the Gromov--Hausdorff distance from a metric space $X$ to simplexes, whose cardinality does not exceed $\#X$. Let $X$ be an arbitrary set and $n$ a cardinal number not exceeding $\#X$. By $\cD_n(X)$ we denote the family of all possible partitions of $X$ into $n$ nonempty subsets.

Now, let $X$ be a metric space. Then for each $D=\{X_i\}_{i\in I}\in\cD_n(X)$ put
$$
\diam D=\sup_{i\in I}\diam X_i.
$$
Notice that $\diam D=\diam X$ for $D\in\cD_1(X)$.

Further, for each $D=\{X_i\}_{i\in I}\in\cD_n(X)$ put
$$
\a(D)=
\begin{cases}
\inf\bigl\{|X_iX_j|:i\ne j\bigr\}&\text{for $n\ge2$,}\\
\infty&\text{for $n=1$.}
\end{cases}
$$

\begin{thm}[\cite{GrigIvaTuzSimpDist}]\label{thm:GenFormulaSmallSympl}
For a metric space $X$, a cardinal number $n\le\#X$ and real $\l\ge0$ the equality
$$
2d_{GH}(\l\D_n,X)=\inf_{D\in\cD_n(X)}\max\Bigl\{\diam D,\l-\a(D),\diam X-\l\Bigr\}
$$
holds.
\end{thm}

Put
$$
\a_n(X)=\sup_{D\in\cD_n(X)}\a(D).
$$
Notice that $\a_1(X)=\infty$.

\begin{cor}\label{cor:BigLambda}
For a metric space $X$, a cardinal number $n\le\#X$ and real $\l$ such that $\l\ge2\diam X$ the equality
$$
2d_{GH}(\l\D_n,X)=
\begin{cases}
\inf_{D\in\cD_n(X)}\bigl(\l-\a(D)\bigr)=\l-\a_n(X)&\text{for $n\ge2$,}\\
\diam X&\text{for $n=1$}
\end{cases}
$$
holds
\end{cor}

\begin{proof}
Since for any partition $D$ the inequalities $\diam D\le \diam X$ and $\a(D)\le\diam X$ are valid, then
$$
\max\Bigl\{\diam D,\l-\a(D),\diam X-\l\Bigr\}=\l-\a(D)
$$
for $\l\ge2\diam X$,
from which the required conclusion follows.
\end{proof}

Let $\CARD$ stands for the proper class of non-zero cardinal numbers. For a metric space $X$ define a mapping  $\cA_X\:\CARD\to[0,\infty]$ as follows: $\cA_X(n)=\a_n(X)$ for $n\le\#X$, and $\cA_X(n)=0$ for $n>\#X$. Theorem~\ref{thm:BigSympl} and Corollary~\ref{cor:BigLambda} imply immediately the following result.

\begin{cor}
If metric spaces $X$ and $Y$ are indistinguishable, then $\cA_X=\cA_Y$.
\end{cor}

\section{$\a$-connected Metric Spaces}
\markright{\thesection.~$\a$-connected Metric Spaces}
A metric space is called \emph{$\a$-connected\/} if for each $n\ge2$ the equality $\a_n(X)=0$ holds. Every connected space is $\a$-connected~\cite{GrigIvaTuzSimpDist}.

Since the transition from a partition to its enlargement does not increase the value of $\a$, the following result is true.	

\begin{prop}[\cite{GrigIvaTuzSimpDist}]\label{prop:cAmonot}
The function $\cA$ is non-increasing.
\end{prop}

\begin{rk}
Proposition~\ref{prop:cAmonot} implies that a metric space $X$ is $\a$-connected, iff $\a_2(X)=0$.
\end{rk}

Let us give another corollary of Theorem~\ref{thm:GenFormulaSmallSympl}. Put
$$
d_n(X)=\inf_{D\in\cD_n(X)}\diam D.
$$
Notice that $d_1(X)=\diam X$.

\begin{cor}\label{cor:GHforAlpha0}
Let $X$ be an arbitrary $\a$-connected metric space, and assume that $n\le\#X$, then
$$
2d_{GH}(\l\D_n,X)=\max\bigl\{d_n(X),\l,\diam X-\l\bigr\}.
$$
\end{cor}

Since $\max\bigl\{\l,\diam X-\l\bigr\}\ge(\diam X)/2$, then these characteristics do not distinguish metric spaces as $d_n(X)\le(\diam X)/2$. For a metric space $X$ we define a mapping $\cD_X\:\CARD\to[0,\infty]$ as follows:
$$
\cD_X(n)=
\begin{cases}
d_n(X)&\text{for $n\le\#X$,}\\
0&\text{for $n>\#X$.}
\end{cases}
$$
Put
$$
\cD_X^{1/2}(n)=
\begin{cases}
\cD_X(n)&\text{for $\cD_X(n)\ge(\diam X)/2$,}\\
0&\text{otherwise.}
\end{cases}
$$

\begin{cor}
Let $X$ and $Y$ be arbitrary $\a$-connected metric spaces. If $X$ and $Y$ are indistinguishable, then $\cD_X^{1/2}=\cD_Y^{1/2}$.
\end{cor}

\section{Distinguishability and $\mst$-spectrum}
\markright{\thesection.~Distinguishability and $\mst$-spectrum}
Let $G=(V,E)$ be an arbitrary graph with vertex set $V$ and edge set $E$. We say that $G$ is \emph{defined on a metric space $X$} if $V\ss X$. For each such graph, define the \emph{lengths $|e|$} of its \emph{edge $e=vw$} as the distances $|vw|$ between the vertices $v$ and $w$ of this edge, and also the \emph{length $|G|$} of the \emph{graph $G$} itself as the sum of the lengths of all its edges.

Let $M$ be a finite metric space. Define the value $\mst(M)$ by setting it equal to the length of the shortest tree among all trees of the form $(M,E)$. The resulting value is called the \emph{length of the minimal spanning tree on $M$}; a tree $G=(M,E)$ such that $|G|=\mst(M)$ is called a \emph{minimal spanning tree on $M$}. Note that a minimal spanning tree exists for any finite metric space $M$. A minimal spanning tree is generally not uniquely defined. The set of all minimal spanning trees on $M$ is denoted by $\MST(M)$.

For $G\in\MST(M)$, let $\s(G)$ denote the vector of edge lengths of $G$, ordered in descending order. The following result is well known, see e.g.~\cite{Emel}.

\begin{prop}
For each $G_1,\,G_2\in\MST(M)$ the equality $\s(G_1)=\s(G_2)$ holds.
\end{prop}

For each finite metric space $M$ by $\s(M)$ we denote the vector $\s(G)$ for arbitrary $G\in\MST(M)$ and call it the  \emph{$\mst$-spectrum of the space $M$}.

\begin{thm}[\cite{TuzMSTSpec}]\label{thm:MSTSpec}
Let $X$ be a finite metric space consisting of $m$ points, $\s(X)=(\s_1,\ldots,\s_{m-1})$, and $\l\ge2\diam X$. Then $2d_{GH}(\l\D_{k+1},X)=\l-\s_k$, where $1\le k\le m-1$.
\end{thm}

\begin{cor}\label{cor:MSTSpect}
Let $X$ and $Y$ be indistinguishable finite metric spaces. Then their cardinalities are the same, and their diameters are equal to each other, and their $\mst$-spectra coincide.
\end{cor}

\begin{proof}
The diameters are equal to each other due to Corollary~\ref{cor:diam}. Let us verify the cardinalities coincidence. Assume the contrary, and let $\#X<\#Y$. Due to Corollary's assumption,  $d_{GH}(\l\D_k,X)=d_{GH}(\l\D_k,Y)$ for any $k$ and any $\l$. Chose $k>\#X$, then in accordance with Theorem~\ref{thm:BigSympl} we have
$$
d_{GH}(\l\D_k,X)=\max\{\l,\diam X-\l\}=\l \qquad\text{for $\l\ge2\diam X$}.
$$
On the other hand, due to Theorem~\ref{thm:MSTSpec}, we have $2d_{GH}(\l\D_{k},Y)=\l-\s_{k-1}(Y)$ for $k\le \#Y$, and hence
$$
\l=d_{GH}(\l\D_k,X)=d_{GH}(\l\D_k,X)=\l-\s_{k-1}(Y),
$$
for $k=\#Y$ and $\l\ge2\diam X$. The latter implies $\s_{k-1}(Y)=0$, a contradiction. Spectra coincidence follows from Theorem~\ref{thm:MSTSpec}.
\end{proof}

\section{Distinguishability of Ultrametric Spaces}
\markright{\thesection.~Distinguishability of Ultrametric Spaces}
Recall that a metric space $X$ is called \emph{ultrametric}, if its distance function satisfies the following enhanced triangle inequality:
$$
|xz|\le\max\bigl\{|xy|,\,|yz|\bigr\}
$$
for each $x,\,y,\,z\in X$. This inequality implies similar  ``polygon inequality'', namely, for an arbitrary set $x_1, x_2,\ldots,x_k$ of points of the space $X$ we have:
$$
|x_1x_k|\le\max\bigl\{|x_1x_2|,|x_2x_3|,\ldots,|x_{k-1}x_k|\bigr\}.
$$
In the paper~\cite{IvaTuzUltra} the next result is obtained.

\begin{thm}\label{thm:ultra}
Let $X$ be a finite ultrametric space consisting of $m$ points, and let $\s(X)=(\s_1,\s_2,\ldots,\s_{m-1})$ be its $\mst$-spectrum. Then for each positive integer $k>0$ we have
$$
2d_{GH}(\l\D_k,X)=
\begin{cases}
\s_1&\text{for $k=1$},\\
\max\{\s_1-\l,\,\s_k,\,\l-\s_{k-1}\}&\text{for $1<k<m$},\\
\max\{\s_1-\l,\,\l-\s_{m-1}\}&\text{for $k=m$},\\
\max\{\s_1-\l,\,\l\}&\text{for $k>m$}.
\end{cases}
$$
\end{thm}

\begin{cor}\label{cor:ultras}
Finite ultrametric spaces are indistinguishable if and only if their cardinalities are the same and their $\mst$-spectra coincide.
\end{cor}

\section{Distinguishability in Terms of Extreme Points}
\markright{\thesection.~Distinguishability in Terms of Extreme Points}
According to Theorem~\ref{thm:GenFormulaSmallSympl}, in order to exactly calculate the function $g_n(\l)=d_{GH}(\l\D_n,X)$ for a given bounded metric space $X$ and cardinal number $n\le\#X$, it is sufficient to know all pairs $\bigl(\a(D),\diam D\bigr)$, $D\in\cD_n(X)$. We consider these pairs as points of the plane on which the standard coordinates $(\a,d)$ are fixed. We denote the set of all such pairs by $\cAD_n(X)\ss\R^2$, and also put $h_{\a,d}(\l):=\max\{d,\,\l-\a\}$. Let us rewrite Theorem~\ref{thm:GenFormulaSmallSympl} in new notations, using the fact that the function $\diam X-\l$ does not depend on the partition $D$, and swapping $\inf$ and $\max$.

\begin{cor}\label{cor:GH-dist-alpha-diam-set}
Let $X$ be an arbitrary bounded metric space, and $n\le\#X$. Then
$$
2d_{GH}(\l\D_n,X)=\max\bigl\{\diam X-\l,\,\inf_{(\a,d)\in\cAD_n(X)}h_{\a,d}(\l)\bigr\}.
$$
\end{cor}

\begin{rk}
For $\a$-connected metric spaces the formula from Corollary~\ref{cor:GH-dist-alpha-diam-set} coincides with the one from Corollary~\ref{cor:GHforAlpha0}
\end{rk}

We need the following notations: for arbitrary metric space $X$ and cardinal number $n\le\#X$, put:
\begin{gather*}
\a_n^-(X)=\inf_{D\in\cD_n(X)}\a(D),\qquad \a_n^+(X)=\sup_{D\in\cD_n(X)}\a(D),\\
d_n^-(X)=\inf_{D\in\cD_n(X)}\diam D,\qquad d_n^+(X)=\sup_{D\in\cD_n(X)}\diam D.
\end{gather*}
Notice that above the notations $\a_n(X)=\a_n^+(X)$ and $d_n(X)=d_n^-(X)$ are used.

\begin{rk}
For $n<\#X$ in the case of a finite metric space, and for $n\le\#X$ in the case of infinite metric space one has $d_n^+(X)=\diam X$.
\end{rk}

In what follows, it will be more convenient to work not with the set $\cAD_n(X)$, but with its closure $\bcAD_n(X)$. Notice that both $\cAD_n(X)$ and $\bcAD_n(X)$ lie in the rectangle formed by the intersection of two strips in the plane $\R^2(\a,d)$: the horizontal strip defined by the inequality $d_n^-(X)\le d\le d_n^+(X)$, and the vertical strip defined by the inequality $\a_n^-(X)\le\a\le\a_n^+(X)$. Thus, both of these sets are bounded, and $\bcAD_n(X)$ is compact.

\begin{cor}\label{cor:GH-dist-alpha-diam-set-closure}
Let $X$ be an arbitrary bounded metric space, and $n\le\#X$. Then
$$
2d_{GH}(\l\D_n,X)=\max\bigl\{\diam X-\l,\,\inf_{(\a,d)\in\bcAD_n(X)}h_{\a,d}(\l)\bigr\}.
$$
\end{cor}

Next, notice that for any pair $(\a,d)$ the graph of the function $y=h_{\a,d}(\l)$ is an angle with its vertex at the point $T_{\a,d}=(\a+d,d)$, one side of which is opposite to the abscissa axis, and the other is co-directed with the bisector of the first quadrant. Notice that for any $(\a,d),\,(\a',d')$ such that $\a\ge\a'$ and $d\le d'$ we have: $h_{\a,d}(\l)\le h_{\a',d'}(\l)$ for all $\l$. Thus, when calculating $d_{GH}(\l\D_n,X)$, such $(\a',d')\in\bcAD_n(X)$ for which there exists $(\a,d)\in\bcAD_n(X)$, $\a\ge\a'$, $d\le d'$, can be ignored.

A point $(\a,d)\in\bcAD_n(X)$ is called \emph{extremal\/} if for it there does not exist a point $(\a',d')\in\bcAD_n(X)$ different from it, for which $\a'\ge\a$ and $d'\le d$. The set of all extremal points from $\bcAD_n(X)$ is denoted by $\Ext_n(X)$. In the paper~\cite{GrigIvaTuzSimpDist} it is shown that the set $\Ext_n(X)$ is not empty, and the following result is true.

\begin{thm}\label{thm:extr}
Let $X$ be an arbitrary bounded metric sp[ace, and $n\le\#X$. Then
$$
2d_{GH}(\l\D_n,X)=\max\bigl\{\diam X-\l,\,\inf_{(\a,d)\in\Ext_n(X)}h_{\a,d}(\l)\bigr\}.
$$
\end{thm}

Notice that the closure of $\overline{\Ext}_n(X)$ is contained in $\bcAD_n(X)$ and contains $\Ext_n(X)$, so from the equality of the left-hand parts of the formulas in Theorem~\ref{thm:extr} and Corollary~\ref{cor:GH-dist-alpha-diam-set-closure} we obtain the following result.

\begin{cor}\label{cor:extr}
Let $X$ be an arbitrary bounded metric sp[ace, and $n\le\#X$. Then
$$
2d_{GH}(\l\D_n,X)=\max\bigl\{\diam X-\l,\,\inf_{(\a,d)\in\overline{\Ext}_n(X)}h_{\a,d}(\l)\bigr\}.
$$
\end{cor}

For a bounded metric space $X$, let $\Pi_X^+\ss\R^2(\a,d)$ denote the half-plane defined by $\a+2d\ge\diam X$, and let $\Pi_X^-\ss\R^2(\a,d)$ denote the half-plane defined by $\a+2d\le\diam X$. Let us define a subset $M_n(X)\ss\R^2$. If $\Ext_n(X)\ss\Pi_X^+$, then we put $M_n(X)=\Ext_n(X)$. If $\Ext_n(X)\ss\Pi_X^-$, then $M_n(X)=\{A_+\}$, where
$$
A_+=\Big(\a_n^+(X),\frac{\diam X-\a_n^+(X)}{2}\Big).
$$
In the remaining case, we denote by $(\a_\pm,d_\pm)$ the point of the set $\overline{\Ext}_n(X)\cap\Pi_X^\pm$ closest to the straight line $\a+2d=\diam X$ (note that in the case under consideration both of these sets are non-empty). If the point $(\a_-,d_+)$ lies in $\Pi_X^-$, then put $M_n(X)=\overline{\Ext}_n(X)\cap\Pi_X^+$. Otherwise $M_n(X)=\overline{\Ext}_n(X)\cap\Pi_X^+\cup\big\{A\big\}$, where
$$
A=\Big(\a_-,\frac{\diam X-\a_-}{2}\Big).
$$

\begin{cor}\label{cor:disting_bounded_Ext}
Bounded metric spaces $X$ and $Y$ of the same cardinality are indistinguishable, if and only if their diameters are the same and the sets $M_n(X)$ and $M_n(Y)$ coincide for each $n\le\#X=\#Y$.
\end{cor}

\begin{proof}
Spaces $X$ and $Y$ are indistinguishable if and only if the functions $f_n(\l)=2d_{GH}(\l\D_n,X)$ and $g_n(\l)=2d_{GH}(\l\D_n,Y)$ are equal for all $n\le\#X=\#Y$ and $\l\ge0$. Let us consider their graphs on the plane with coordinates $(\l,y)$. It follows from Theorem~\ref{thm:extr} that the graph of the function $f_n(\l)=2d_{GH}(\l\D_n,X)$ is a broken line, one link of which is a segment of the straight line $y=\diam X-\l$, and the remaining links lie on the rays of the graphs of the functions $h_{\a,d}$, where $(\a,d)\in\overline{\Ext}_n(X)$. Obviously, functions $f_n$ and $g_n$ are equal if and only if their graphs coincide, and the latter takes place if and only if the vertex sets of the corresponding broken lines coincide.

Let us describe the graph of function $f_n(\l)$. Denote by $L_{\a,d}$ the graph of function $h_{\a,d}$. Recall that this is a broken line consisting of two rays with a vertex at the point with coordinates $(d+\a,d)$. Since in the formula from Theorem~\ref{thm:extr} the maximum of the greatest lower bound of functions $h_{\a,d}$ and function $\diam X-\l$ is taken, then on the graph of function $f_n(\l)$ there are those and only those extremal points that lie in the half-plane $\Pi_X^+$. If all extreme points fall in this half-plane, then the set $M_n(X)=\Ext_n(X)$ completely determines this graph. If $\Ext_n(X)\ss\Pi_X^-$, then the graph of the function $f_n(\l)$ consists of two links lying on the straight lines $y=\diam X-\l$ and $y=\l-\a_n^+(X)$. These two straight lines on the plane with coordinates $(\l,y)$ intersect at the point with coordinates
$$
\Big(\frac{\diam X+\a_n^+(X)}{2},\frac{\diam X-\a_n^+(X)}{2}\Big),
$$
that corresponds to the point
$$
A_+=\Big(\a_n^+(X),\frac{\diam X-\a_n^+(X)}{2}\Big)
$$
at the plane with coordinates $(\a,d)$, $\a=\l-y$, $d=y$. Thus, in this case the set $M_n(X)=\{A_+\}$ completely defines the graph of the function $f_n(\l)$ and conversely, $M_n(X)$ can be restored from the graph.

Finally, let the extreme points be located in both half-planes. In this situation, the lower envelope $E_L$ of the family of broken lines $\{L_{\a,d}\}$ intersects the line $y=\diam X-\l$ at some point. To find this point, consider the vertices of the envelope $E_L$ closest to this line. These are the points $B_\pm=(\a_\pm+d_\pm,d_\pm)\in\R^2(\l,y)$, where $(\a_\pm,d_\pm)\in\Pi_X^\pm$. The part of $E_L$ between these two points is a broken line consisting of a horizontal straight segment with endpoint at $B_+$ and an inclined straight segment with endpoint at $B_-$. The common vertex of these straight segments is $(\a_-+d_+,d_+)$. On the plane $\R^2(\a,d)$ this point has the form $(\a_-,d_+)$. If this point lies in $\Pi_X^-$, then the line $y=\diam X-\l$ intersects the horizontal segment of the lower envelope and the coordinates of the intersection point are calculated through $(\a_+,d_+)\in\Pi_X^+$ and $\diam X$. If the point $(\a_-,d_+)$ lies in $\Pi_X^-$, then the line $y=\diam X-\l$ intersects the inclined segment of the envelope $E_L$, and the coordinates of the intersection point have the form
$$
\Big(\frac{\diam X+\a_-}{2},\frac{\diam X-\a_-}{2}\Big),
$$
that corresponds to the point
$$
A=\Big(\a_-,\frac{\diam X-\a_-}{2}\Big).
$$
in the plane $\R^2(\a,d)$. Thus, in this case also the set $M_n(X)$ completely determines the graph of the function $f_n(\l)$ and, conversely, is can be reconstructed from the graph. The corollary is proved.
\end{proof}

Notice that indistinguishable bounded metric spaces could have different extreme sets.

\begin{examp}
Let $X$ be a subset of the plane consisting of four equal straight segments of length $s$ lying on the same straight line $\ell$, and $\dl$ be the distance between the nearest endpoints of adjacent segments. Let us calculate the distance $d_{GH}(X,\l\D_4)$. It is clear that $\diam X=4s+3\dl$, $d_4^+(X)=\diam X$, $d_4^-(X)=s$, $\a_4^-(X)=0$, $\a_4^+(X)=\dl$, there is one extremal point with coordinates $(\dl,s)$ on the plane $\R^2(\a,d)$ or $(s+\dl,s)$ in coordinates $(\l,y)$. This point lies in the half-plane $\Pi_X^-$, and the set $M_4(X)$ consists of a single point $A_+(X)$. Let us change the space $X$, preserving the diameter and $\a_4^+$, and changing $d_4^-$. To do this, let us rotate the segments with respect to their midpoints and extend them so that their ends remain on straight lines perpendicular to the straight line $\ell$ and passing through the endpoints of the segments forming the space $X$, see Fig.~\ref{fig:X_Y}. We denote the resulting space by $Y$.

\ig{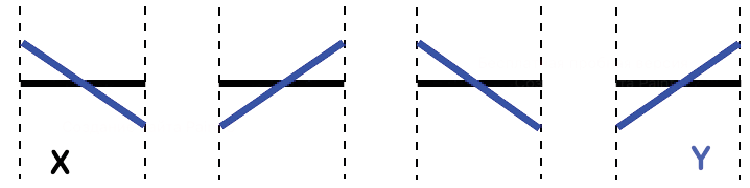}{0.15}{fig:X_Y}{Indistinguishable metric spaces $X$ and $Y$ with different sets $\Ext_4$.}

Then $\diam Y=\diam X$, $d_4^+(Y)=\diam Y$, $d_4^-(Y)=S>s$, $\a_4^-(Y)=0$, $\a_4^+(Y)=\dl$, the space $Y$ has one extremal point with coordinates $(\dl,S)$ on the plane $\R^2(\a,d)$ or $(S+\dl,S)$ in coordinates $(\l,y)$. For $S<\dl+2s$ this point lies in the half-plane $\Pi_Y^-=\Pi_X^-$, and the set $M_4(Y)$ consists of one point $A_+(Y)=A_+(X)$. Thus, in the notation of the proof of Corollary~\ref{cor:disting_bounded_Ext}, the graphs of the functions $f_4(\l)$ and $g_4(\l)$ coincide, and the sets $\Ext_4(X)$ and $\Ext_4(Y)$ are different, see Fig.~\ref{fig:graph_f_g}.

\ig{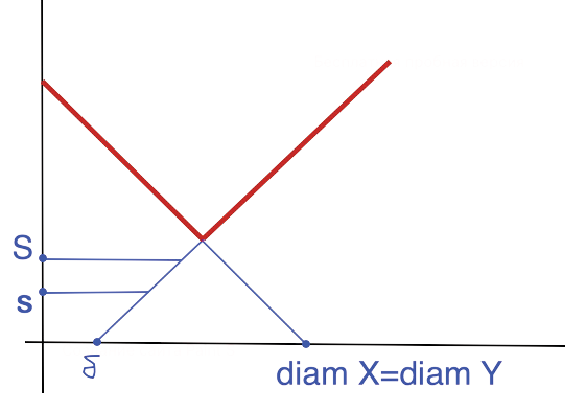}{0.25}{fig:graph_f_g}{The graphs of the functions $f_4(\l)=2d_{GH}(\l\D_4,X)$ and $g_4(\l)=2d_{GH}(\l\D_4,Y)$ coincide.}
\end{examp}

Combining the Corollaries~\ref{cor:MSTSpect} and~\ref{cor:disting_bounded_Ext}, we obtain the following indistinguishability criterion for finite metric spaces.

\begin{cor}\label{cor:finite_disting_bounded_Ext}
Finite metric spaces $X$ and $Y$ are indistinguishable, if and only if their diameters are equal and the subsets $M_n(X)$ and $M_n(Y)$ coincide for each  $n\le\#X=\#Y$.
\end{cor}

\section{Distinguishability and Graph Chromatic Number}
\markright{\thesection.~Distinguishability and Graph Chromatic Number}
In paper~\cite{IvaTuzTwoDist} a connection was found between the Gromov--Hausdorff distances from metric spaces with two nonzero distances to simplexes and the chromatic numbers of graphs. Recall that the \emph{chromatic number of a simple graph $G$} is the smallest number of colors in which the vertices of this graph can be colored so that adjacent vertices are colored in different colors. We denote this number by $\g(G)$.

\begin{thm}[\cite{IvaTuzTwoDist}]\label{thm:ChromNum}
Let $G=(V,E)$ be an arbitrary finite simple graph, where $V$ stands for its vertex set and $E$ stands for its edge set. Fix arbitrary reals $a$ and $b$ such that $0<a<b\le2a$ and define a metric on $V$ as follows\/\rom: put the distances between each adjacent vertices of $G$ to be equal to $b$, and put all the remaining non-zero distances to be equal to $a$. If $m$ is the largest positive number such that $2d_{GH}(a\D_m,V)=b$, then $\g(G)=m+1$.
\end{thm}

For $0<a<b\le2a$, a metric space in which nonzero distances take only values ??$a$ or $b$ is called an \emph{$(a,b)$-space}. Let $X$ be an arbitrary $(a,b)$-space. Construct a graph $G_X=(V,E)$ by reversing the procedure from Theorem~\ref{thm:ChromNum}, namely, we set $V=X$ and connect vertices $x,y\in X$ by an edge if and only if $|xy|=b$. From Corollary~\ref{cor:MSTSpect} and Theorem~\ref{thm:ChromNum}, we obtain the following result.

\begin{cor}
Let $X_i$, $i=1,\,2$, be some $(a_i,b_i)$-spaces and let $0<a_1<b_1\le2a_1$, $0<a_2<b_2\le2a_2$. Assume that $X_1$ and $X_2$ are indistinguishable. Then they have the same cardinality, $a_1=a_2$, $b_1=b_2$, and the chromatic numbers of the graphs $G_{X_1}$ and $G_{X_2}$ coincide.
\end{cor}

\section{Distinguishability and Graph Clique Covering Number}
\markright{\thesection.~Distinguishability and Graph Clique Covering Number}
In the same paper~\cite{IvaTuzTwoDist} a connection is found between the Gromov--Hausdorff distances from metric spaces with two nonzero distances to simplexes and the clique covering numbers of graphs. The point is that the clique covering number and the chromatic number are dual, see, for example~\cite{Emel}. Let us recall the corresponding definitions.

A \emph{clique\/} in an arbitrary simple graph $G=(V,E)$ is each of its subgraphs in which every two vertices are connected by an edge. Notice that every single-vertex subgraph is a clique by definition. It is clear that the family of vertices of all cliques covers $V$. The smallest number of cliques whose vertices cover $V$ is called the \emph{clique covering number}. We denote this number by $\theta(G)$.

For a simple graph $G$, let $G'$ stand for its \emph{dual\/} graph, i.e. the graph that has the same set of vertices and the complementary set of edges (two vertices in $G'$ are connected by an edge if and only if they are not connected by an edge in $G$).

The following fact is well known.

\begin{prop}
For each simple graph $G$ the equality $\theta(G)=\g(G')$ holds.
\end{prop}

Theorem~\ref{thm:ChromNum} can be reformulated in the following form.

\begin{thm}[\cite{IvaTuzTwoDist}]\label{thm:clique}
Let $G=(V,E)$ be an arbitrary finite simple graph, where $V$ is its vertex set, and $E$ is its edge set. Fix arbitrary  $0<a<b\le2a$ and define a metric on $V$ as follows\/\rom: put the distances between adjacent vertices to be equal to $a$, and put all the remaining nonzero distances to be equal to $b$. If $n$ is the largest positive integer such that  $2d_{GH}(a\D_n,V)=b$, then $\theta(G)=n+1$.
\end{thm}

Let $0<a<b\le2a$ and $X$ be an arbitrary $(a,b)$-space. Construct the graph $H_X=(V,E)$ by reversing the procedure from Theorem~\ref{thm:clique}, namely, set $V=X$ and connect vertices $x,y\in X$ by an edge if and only if $|xy|=a$. From Corollary~\ref{cor:MSTSpect} and Theorem~\ref{thm:clique}, we immediately obtain the following result.

\begin{cor}
Let $X_i$,  $i=1,\,2$,  be some $(a_i,b_i)$-spaces and $0<a_1<b_1\le2a_1$, $0<a_2<b_2\le2a_2$. Assume that  $X_1$ and $X_2$ are indistinguishable. Then they have the same cardinality, $a_1=a_2$, $b_1=b_2$, and the clique covering numbers of the graphs $H_{X_1}$ and $H_{X_2}$ coincide.
\end{cor}

\section{Distinguishability and the Borsuk Number}
\markright{\thesection.~Distinguishability and the Borsuk Number}
In this section we discuss the connection between distinguishability and the famous Borsuk problem: Into how many parts should an arbitrary non-single-point subset of the Euclidean space $\R^n$ be partitioned so that the diameters of the partition elements are smaller than the diameter of the initial set. In~1933, K.~Borsuk stated the famous conjecture that any bounded subset of $\R^n$ can be partitioned into $(n+1)$ subsets of smaller diameter. This conjecture was proved by Hadwiger~\cite{Hadw} and~\cite{Hadw2} for convex subsets with smooth boundary, and unexpectedly disproved in the general case in~1993~\cite{KahnKalai}. The current state of the art is described, for example, in~\cite{Raig}.

In~\cite{IvaTuzBorsuk} the following generalization of the Borsuk problem is formulated. Let $X$ be a bounded metric space, $n$ a cardinal number, $n\le\#X$, and $D=\{X_i\}_{i\in I}\in\cD_n(X)$. We say that $D$ is a partition into parts \emph{of strictly smaller diameter} if there exists $\e>0$ such that $\diam X_i\le\diam X-\e$ for all $i\in I$. The following problem is called the \emph{generalized Borsuk problem\/}: find out whether a given bounded metric space can be partitioned into a given number of parts of strictly smaller diameter. It turns out that the answer can be obtained in terms of the Gromov--Hausdorff distance to suitable simplexes.

\begin{thm}[\cite{IvaTuzBorsuk}]
Let $X$ be an arbitrary bounded metric space and $n$ a cardinal number such that $n\le\#X$. Chose an arbitrary real $\l$,  $0<\l<\diam X$, then $X$ can be partitioned into $n$ parts of strictly smaller diameter if and only if $2d_{GH}(\l\D_n,X)<\diam X$.
\end{thm}

\begin{cor}
Let $X$ and $Y$ be indistinguishable metric spaces. Then for each cardinal number $n$, $n\le\min\{\#X,\#Y\}$, the spaces $X$ and $Y$ either simultaneously allow or do not allow a partition  into $n$ parts of strictly smaller diameter.
\end{cor}

\section{Examples of Indistinguishable Spaces}
\markright{\thesection.~Examples of Indistinguishable Spaces}

In conclusion, let us give some examples of indistinguishable spaces.

\subsection{Indistinguishable three-point and continuum spaces}
Let $M$ be a three-point space with distances $a<b<c$, and let $X$ be a disjoint union of connected bounded subsets $X_1$, $X_2$, $X_3$ such that the diameters of all the three ones are $d<(c-a)/2$, and the distances between any points from $X_1$ and $X_2$, $X_1$ and $X_3$, $X_2$ and $X_3$ are $a$, $b$, $c$, respectively. Then
$$
d_{GH}(\l\D_n,M)=d_{GH}(\l\D_n,X).
$$
Indeed, for $n>\#X$ this implies from Theorem~\ref{thm:BigSympl}. It remains to verify the case $n\le\#X$.

First, let $n>3$. Then each $D\in\cD_n(X)$ partitions one of the connected spaces $X_i$, and therefore $\a(D)=0$. By the corollary~\ref{cor:GHforAlpha0}, we have
$$
2d_{GH}(\l\D_n,X)=\max\Bigl\{d_n(X),\l,\diam X-\l\Bigr\}.
$$
On the other hand, $\max\{\l,\diam X-\l\}\ge(\diam X)/2=c/2$, and for $D\in\cD_n(X)$, each element of which lies in some $X_i$, we have $\diam D\le(c-a)/2<c/2$, therefore $d_n(X)<c/2$ and, therefore, $2d_{GH}(\l\D_n,X)=\max\{\l,\diam X-\l\}$. By Theorem~\ref{thm:BigSympl}, the same formula holds for $d_{GH}(\l\D_n,M)$.

Let $n=3$. Let $D_0=\{X_1,X_2,X_3\}$, then for each $D\in\cD_3(X)$, $D\ne D_0$, we have $\a(D)=0$ again. Define $d'_n(X)=\inf_{D\ne D_0}\diam D$, then $d'_n(X)<c/2$ again and
$$
\inf_{D\ne D_0}\max\Bigl\{\diam D,\l,\diam X-\l\Bigr\}=\max\{\l,\diam X-\l\}=\max\{\l,c-\l\}.
$$
On the other hand, $\a(D_0)=a$, therefore
$$
\max\Bigl\{\diam D_0,\l-\a(D_0),\diam X-\l\Bigr\}=\max\{d,\l-a,c-\l\}.
$$
Given that $d<(c-a)/2$ and $\max\{d,\l-a,c-\l\}\ge(c-a)/2$, obtain $2d_{GH}(\l\D_n,X)=\max\{\l-a,c-\l\}$. Finally, in the paper~\cite{IvaTuzIrreduce} it is shown that the Gromov--Hausdorff distance between two three-point spaces can be calculated as the $\max$-norm of the difference of ordered distance vectors. Therefore, the same formula holds for $d_{GH}(\l\D_n,M)$.

Let $n=2$. Considering $\l\D_2$ as a degenerated three-point space with the distances $0$, $\l$, $\l$ and applying the same result from~\cite{IvaTuzIrreduce} conclude that
$$
2d_{GH}(\l\D_2,M)=\max\bigl\{a,|b-\l|,|c-\l|\bigr\}.
$$
Let us show that the same formula is valid for the space $X$. Notice that for each partition $D\in\cD_2(X)$ we have $\diam D\ge a$, because one element of this partition contains points from at least two sets $X_i$, of which $X$ consists. Besides, $\a(D)\le b$, because either the partition $D$ splits one of $X_i$ and $\a(D)=0$, or it does not split any $X_i$, and one of its elements is some $X_i$, and the other one is the union of the remaining two $X_j$. So,
$$
2d_{GH}(\l\D_2,X)\ge\max\big\{a,|\l-b|,|c-\l|\big\}.
$$
Notice that this estimate is attained at the partition $D=\{X_1\cup X_2,X_3\}$, that completes the proof.

\subsection{Indistinguishable connected spaces}
Let $X=[0,1]$ be a straight segment and $Z$ a connected metric space, $\diam Z\le1/2$. Put $Y=X\x Z$ and define a distance function on $Y$ as follows:
$$
\bigl|(x_1,z_1)(x_2,z_2)\bigr|=\max\bigl\{|x_1x_2|,|z_1z_2|\bigr\}.
$$
The space $Y$ is connected and $\diam X=\diam Y=1$, therefore
$$
2d_{GH}(\l\D_n,X)=\max\bigl\{d_n(X),\l,1-\l\bigr\},\qquad
2d_{GH}(\l\D_n,Y)=\max\bigl\{d_n(Y),\l,1-\l\bigr\}.
$$
Notice that $d_2(X)\le1/2$, $d_2(Y)\le1/2$, and $\max\{\l,1-\l\}\ge1/2$, therefore using monotony of $d_n$ conclude that $d_n(X)\le1/2$ and $d_n(Y)\le1/2$ for all $n\ge2$, and hence $d_{GH}(\l\D_n,X)=d_{GH}(\l\D_n,Y)$ for $n\ge2$. Since $\diam X=\diam Y$, then $d_{GH}(\D_1,X)=d_{GH}(\D_1,Y)$ also, so the spaces $X$ and $Y$ are indistinguishable.

Let us generalize this example as follows.

\begin{thm}
Let $X$ and $Y$ be $\a$-connected metric spaces of the same diameter, and assume that $d_2(X)\le(\diam X)/2$ and $d_2(Y)\le(\diam Y)/2$. Then $X$ and $Y$ are indistinguishable. In particular, for any cardinal numbers greater than or equal to continuum there exists indistinguishable metric spaces of such cardinalities.
\end{thm}

\markright{References}


\begin{thebibliography}{20}
\bibitem{Hausdorff} F.~Hausdorff,  {\it Grundz\"uge der Mengenlehre},  Veit, Leipzig, 1914 [reprinted by Chelsea in 1949].

\bibitem{Edwards} D.~Edwards, \emph{The Structure of Superspace}, in {\it Studies in Topology}, Academic Press, 1975.

\bibitem{Gromov} M.~Gromov, {\emph Groups of Polynomial Growth and Expanding Maps}, Publications Mathematiques I.H.E.S., {\bf 53}, 1981.

\bibitem{BurBurIva} D.~Burago, Yu.~Burago, S.~Ivanov, \emph{A Course in Metric Geometry}, Graduate Studies in Mathematics {\bf 33} A.M.S., Providence, RI, 2001.

\bibitem{IvaNikTuz} Ivanov A.\,O., Nikolaeva N.\,K., Tuzhilin A.\,A. \emph{The Gromov-Hausdorff Metric on the Space of Compact Metric Spaces is Strictly Intrinsic}.  ArXiv e-prints: {\tt arXiv:1504.03830}, 2015.

\bibitem{GrigIvaTuzSimpDist} 
D.~S.~Grigor'ev, A.~O.~Ivanov and A.~A.~Tuzhilin, \emph{Gromov--Hausdorff Distance to Simplexes\textquotedblright}, Chebyshevski Sbornik, \textbf{20} (2), 100--114 (2019).
l
\bibitem{IvaTuzDistSympl} A.\,O.~Ivanov, A.\,A.~Tuzhilin, \emph{Geometry of Compact Metric Space in Terms of Gromov--Hausdorff Distances to Regular Simplexes}, ArXiv e-prints, {\tt arXiv:1607.06655}, 2016.

\bibitem{TuzMSTSpec} A.\,A.~Tuzhilin, \emph{Calculation of Minimum Spanning Tree Edges Lengths using Gromov--Hausdorff Distance}, ArXiv e-prints, {\tt arXiv:1605.01566}, 2016.

\bibitem{IvaTuzTwoDist} A.\,O.~Ivanov, A.\,A.~Tuzhilin, \emph{The Gromov-Hausdorff Distance between Simplexes and Two-Distance Spaces},  ArXiv e-prints, {\tt arXiv:1907.09942}, 2019.

\bibitem{Hadw} H.~Hadwiger, \emph{\"Uberdeckung einer Menge durch Mengen kleineren Durchmessers}. Commentarii Mathematici Helvetici, v. 18 (1): 73--75, (1945).

\bibitem{Hadw2} H.~Hadwiger, \emph{Mitteilung betreffend meine Note: \"Uberdeckung einer Menge durch Mengen kleineren Durchmessers}. Commentarii Mathematici Helvetici, v. 19, (1946).

\bibitem{KahnKalai} J.~Kahn, G.~Kalai, \emph{A counterexample to Borsuk’s conjecture}. Bull. Amer. Math. Soc., v. 29 (1), 60--62 (1993).

\bibitem{Raig} A.~M.~Raigorodskii, \emph{Around Borsuk's Hypothesis}, Journal of Mathematical Sciences, \textbf{154} (4)  604--623 (2008).

\bibitem{IvaTuzBorsuk} A.~O.~Ivanov and A.~A.~Tuzhilin, \emph{Calculation of the Gromov–Hausdorff distance using the Borsuk number}, Moscow University Mathematics Bulletin \textbf{78} (1) 37--43 (2023).

\bibitem{IvaTuzIrreduce} A.\,O.~Ivanov, A.\,A.~Tuzhilin, \emph{Gromov--Hausdorff Distance, Irreducible Correspondences, Steiner Problem, and Minimal Fillings},  ArXiv e-prints, {\tt arXiv:1604.06116}, 2016.

\bibitem{IvaTuzUltra} Ivanov A. O., Tuzhilin A. A. \emph{The Gromov--Hausdorff distances between simplexes and ultrametric spaces}, ArXiv e-prints, {\tt arXiv:1907.03828}, 2019.

\bibitem{Emel} V.~A. Emelichev, O.~I. Mel'nikov, V.~I. Sarvanov, and R.~I. Tyshkevich, {\emph Lectures on Graph Theory}, Moscow, URSS, 2019.
\end{thebibliography}
\end{document}